\documentclass[11pt]{amsart}
\usepackage{amsmath,amsthm,amssymb,amssymb, paralist, xspace, graphicx, url, amscd, euscript, mathrsfs,stmaryrd,epic,eepic,color}
\usepackage[all]{xy}
\usepackage{amsmath,amscd}
\SelectTips{cm}{}

\usepackage[colorlinks=true]{hyperref}

\numberwithin{equation}{subsection}

1

\numberwithin{subsection}{section}

\allowdisplaybreaks[1]

\newtheorem*{namedtheorem}{\theoremname}
\newcommand{\theoremname}{testing}

\theoremstyle{plain}

\newtheorem{thm}{Theorem}[section]
\newtheorem{proposition}[thm]{Proposition}
\newtheorem{proposition-definition}[thm]{Proposition-Definition}
\newtheorem{lemma-definition}[thm]{Lemma-Definition}
\newtheorem{corollary}[thm]{Corollary}
\newtheorem{lemma}[thm]{Lemma}

\theoremstyle{definition}
\newtheorem{definition}[thm]{Definition}
\newtheorem{notation}[thm]{Notation}

\newtheorem{remark}[thm]{Remark}

\newtheorem{question}[thm]{Question}
 
\newtheorem{construction}[thm]{Construction}

\newtheorem{obs}[thm]{Observation}

\theoremstyle{remark}

\numberwithin{thm}{section}

\newcommand\ocM{\overline{\mathcal{M}}}

\newcommand\cC{\mathcal{C}}

\newcommand\cL{\mathcal{L}}
\newcommand\cM{\mathcal{M}}

\newcommand\cO{\mathcal{O}}

\newcommand\cT{\mathcal{T}}

\newcommand\cX{\mathcal{X}}

\newcommand\E{\mathcal{E}}
\newcommand\K{\mathcal{K}}
\newcommand\N{\mathcal{N}}
\newcommand\T{\mathcal{T}}
\def\O{\mathcal{O}}

\def\P{\mathbb{P}}
\def\A{\mathbb{A}}

\def\N{\mathbb{N}}
\def\F{\mathcal{F}}

\newcommand\uf{\underline{f}}

\renewcommand\AA{\mathbb{A}}

\newcommand\NN{\mathbb{N}}
\newcommand\PP{\mathbb{P}}

\newcommand\fK{\mathfrak{K}}

\newcommand\fM{\mathfrak{M}}

\newcommand\kk{\mathbf{k}}  %%% base field

\newcommand\arr{\ifinner\to\else\longrightarrow\fi}

\def\displaytimes_#1{\mathrel{\mathop{\times}\limits_{#1}}}

\def\displayotimes_#1{\mathrel{\mathop{\bigotimes}\limits_{#1}}}

\renewcommand\hom{\operatorname{\underline {Hom}}}

\newcommand\ext{\operatorname{Ext}}

\newcommand\spec{\operatorname{Spec}}

\newcommand{\Tor}{\operatorname{Tor}}

\newcommand\ch{\operatorname{char}}

\newdir{ >}{{}*!/-5pt/@{>}}

\newcommand\doublelong[2]{\mathbin{\xymatrix{{}\ar@<3pt>[r]^{#1}
\ar@<-3pt>[r]_{#2}&}}}

\newlength{\ignora}

\renewcommand{\setminus}{\smallsetminus}

\newcommand{\Hom}{\underline {\text{\rm Hom}}}

\numberwithin{equation}{subsection}

\newcommand{\Jac}{\operatorname{Jac}}

\newcommand{\diag}{\operatorname{diag}}

\newcommand{\lX}{X^{\dagger}}
\newcommand{\lY}{{Y^{\dagger}}}
\newcommand{\lC}{C^{\dagger}}
\newcommand{\lcX}{{\cX^{\dagger}}}
\newcommand{\lB}{{B^{\dagger}}}
\newcommand{\lS}{{S^{\dagger}}}

\begin{document}

\title{Very Free Curves on Fano Complete Intersections}

\author{Qile Chen}

\author{Yi Zhu}

\thanks{Chen is supported by the Simons foundation.}

\address[Chen]{Department of Mathematics\\
Columbia University\\
Rm 628, MC 4421\\
2990 Broadway\\
New York, NY 10027\\
U.S.A.}
\email{q\_chen@math.columbia.edu}

\address[Zhu]{Department of Mathematics\\
University of Utah\\
Room 233\\
155 S 1400 E \\
Salt Lake City, UT 84112\\
U.S.A.}
\email{yzhu@math.utah.edu}

\date{\today}
\begin{abstract}
In this paper, we show that general Fano complete intersections over an algebraically closed field of arbitrary characteristics are separably rationally connected. Our construction of rational curves leads to a more interesting generalization  that general log Fano complete intersections with smooth tame boundary divisors admit very free $\AA^1$-curves.
\end{abstract}
\maketitle

\tableofcontents

%%%%%%%%%%%%%%%%%%%%%%%%%%%%%%%%%%%%%%%%%%%%%%%%%%%%%%%%%%%%%%%%%%%%%%%%%%%%%%%%%%%%%%%%%%%%%%%%%%%
\section{Introduction}
%%Introduction

%---------------------Back ground and main result

\subsection{The back ground and main results}

The existence of rational curves in higher dimensional varieties greatly shapes the geometry. The following definitions describe the existence of large amount of rational curves:

\begin{definition}[\cite{Kollar} IV.3]\label{def:rc}
Let $X$ be a variety defined over an arbitrary field $k$. 
\begin{enumerate}
\item A variety $X$ is \emph{rationally connected} (RC) if there is a family of irreducible proper rational curves $g: U\rightarrow Y$ and the cycle morphism $u:U\rightarrow X$ such that the morphism $u^{(2)}:U\times_Y U\rightarrow X\times X$ is dominant.

\item A variety $X$ is \emph{rationally chain connected} (RCC) if there is a family of chains of rational curves $g: V\rightarrow Y$ and the cycle morphism $u:V\rightarrow X$ such that the morphism $u^{(2)}:V\times_Y V\rightarrow X\times X$ is dominant.

\item A variety $X$ is \emph{separably rationally connected} (SRC) if over the algebraic closure $\overline{k}$, there exists a proper rational curve $f:\P^1\rightarrow X $ such that $X$ is smooth along the image and $f^*T_X$ is ample.
%the image lies in the smooth locus of $X$ and the pullback of the tangent sheaf $f^*TX$ is ample. 
Such rational curves are called \emph{very free} curves.
\end{enumerate}
\end{definition}

We refer to Koll\'ar's book \cite{Kollar} for the background. The third definition is stronger than \cite[IV3.2.3]{Kollar}.  But they coincide for smooth varieties. It is known that SRC implies RC, and RC implies RCC. All the notations of rational connectedness are equivalent for smooth varieties in characteristic zero. However in positive characteristics, it is known that RC is strictly weaker than SRC.

The fundamental results of Campana \cite{Campana} and Koll\'ar, Miyaoka and Mori \cite{KMM} show that Fano varieties, i.e., smooth varieties with ample anticanonical bundles are rationally chain connected. In particular, Fano varieties are SRC in characteristic zero. 

It has been pointed out by Koll\'ar that separable rational connectedness is the right notion for rational connectedness in arbitrary characteristics. They have both nice geometric and arithmetic applications as follows.

Geometrically, when the base field is algebraically closed, \begin{enumerate}
\item Graber-Harris-Starr \cite{GHS} proves the famous theorem asserting that over characteristic zero, a proper family of varieties over an algebraic curve whose general fiber is smooth RC admits a section;

\item de Jong and Starr in \cite{dJS1} generalize the result of \cite{GHS} by showing the existence of sections over positive characteristics when general fibers are smooth SRC. 

\item The weak approximation for families of separably rationally connected varieties was studied by \cite{HT06}, \cite{HT08}.

\item Tian and Zong \cite{Tian-Zong} show that the Chow group of $1$-cycles on a smooth proper SRC variety is generated by rational curves. 
\end{enumerate}

Arithmetically, we have: 
\begin{enumerate}
\item when the base field is a local field, Koll\'ar \cite{Kollar-local} shows that a smooth proper SRC variety admits a very free curve through any rational point;

\item when the base field is a finite field of cardinality $q \leq \infty$, Koll\'ar and Szab\'o \cite{Kollar-Szabo} show that there is a function $\Phi:\N^3\to N$ such that for a smooth projective SRC variety $X \subset \PP^N$, given any zero dimensional subscheme $S\subset X$, there exists a smooth rational curve on $X$ containing $S$ whenever $q>\Phi(\deg X,\dim X, \deg S)$;

\item when the base field is a large field, Hu \cite{HuYong} proves interesting results on the weak approximation conjecture for SRC varieties at places of good reductions.
\end{enumerate}

Despite the nice behavior of SRC varieties, many important varieties which are known to be RC over characteristic zero, are difficult to verify the SRC condition in positive characteristics. The following question is the major motivation of the present paper:

\begin{question}[Koll\'ar] \label{q:K}
In arbitrary characteristic, is every smooth Fano variety separably rationally connected?
\end{question}

\begin{notation}
Since Question \ref{q:K} can be checked over the algebraic closure, for the rest of this paper, we work with algebraic varieties over an algebraically closed field $\kk$ of arbitrary characteristics.  
\end{notation}

The first testing example is Fano complete intersections in projective spaces. The difficulty is to prove separable rational connectedness in low characteristics \cite[Conjecture 14]{Kollar-Szabo}. 

The question is known for general Fano hypersurfaces by \cite{Zhu}, where very free curves are constructed explicitly over degenerate Fano varieties. In this paper, we provide an answer in the complete intersection case:

\begin{thm}\label{thm:main-fano}
Over arbitrary characteristics, a general Fano complete intersection in $\PP^n$ is separably rationally connected. %{\red estimate the degree?}
\end{thm}

\begin{remark}
During the preparation of this paper, the authors learned another interesting proof of of Theorem \ref{thm:main-fano} by Zhiyu Tian \cite{Tian-13} using a different method. 
\end{remark}

Our theorem eliminates the SRC condition in \cite[Theorem 1.7]{Tian-Zong}.

\begin{corollary}
Let $X$ be a general complete intersection of type $(d_1,\cdots,d_l)$ in $\P^n$ such that $d_1+\cdots+d_l\le n-1$. Then the Chow group of $1$-cycles on $X$ is generated by lines. 
\end{corollary}

Surprisingly, our construction of very free curves leads to a much stronger generalization of Theorem \ref{thm:main-fano}:

\begin{thm}\label{thm:main-log-fano}
Over arbitrary characteristics, a general log Fano complete intersection with a general tame boundary in $\PP^n$ is separably $\A^1$-connected. 
\end{thm}

We refer to Section \ref{ss:log-pair}, and \ref{ss:log-rc} for more details of the definitions and the proof of the above results. 

Note that the $\AA^1$-connectedness implies rationally connectedness of the underlying variety. However, the other direction fails in general. In fact, a log Fano log variety needs not to be $\AA^1$-connected! For example the log variety associated to $\PP^2$ with boundary given by two distinct lines fails to be $\AA^1$-connected. On the other hand, producing an $\AA^1$-curve is much more difficult than producing rational curves due to the constrain on the boundary marking. Thus, Theorem \ref{thm:main-log-fano} turns out to be a more interesting result.

\begin{question}
Can we drop the tame condition on the boundary in Theorem \ref{thm:main-log-fano}? 
\end{question}

Our construction of very free log maps provides a much more general result over characteristic zero:

\begin{thm}\label{thm:main-log-fano-char-zero}
Assume $\ch \kk = 0$. Let $X$ be a log Fano smooth pair,  i.e., $-(K_X+D)$ is ample, with a smooth irreducible boundary divisor $D$. Then $X$ is separably $\AA^1$-connected if and only if it is separably $\AA^1$-uniruled.
\end{thm}

\begin{corollary}\label{cor}
Let $X$ be a log Fano smooth pair as in Theorem \ref{thm:main-log-fano-char-zero}, and $\ch \kk = 0$. If the divisor class of the normal bundle of $D$ is numerically equivalent to a nontrivial effective divisor, then $(X,D)$ is $\A^1$-connected.
\end{corollary}

The proof of Theorem \ref{thm:main-log-fano-char-zero} and Corollary \ref{cor} will be given in the end of Section \ref{ss:log-fano}. 

The condition in above corollary is first stated in the work of Hassett-Tschinkel \cite{HT-log-Fano}. The $\A^1$-uniruledness condition seems to be a more natural setting --- it includes for example $(\P^1, \{\infty\})$, or Hirzebruch surfaces with the negative curve as the boundary, which are $\A^1$-uniruled, but the divisor class of the normal bundle of the boundary divisor is either trivial or non-effective.

Keel-McKernan \cite{KM} prove that any log Fano pair over complex numbers is either $\A^1$-uniruled or uniruled. It is natural to ask:

\begin{question}
Let $(X,D)$ be a log smooth log Fano variety with $D$ irreducible. Is the pair $(X,D)$ always $\A^1$-uniruled?
\end{question}

\begin{remark}
It should be emphasized that our proof of Theorem \ref{thm:main-fano} and \ref{thm:main-log-fano} is constructive, which allows one to write down the exact degree of the very free curves in each case. We leave the details of this to interested readers.

On the other hand, it seems to us that $\AA^1$-connectedness itself is a very useful conception for the study of quasi-projective varieties. The results of Theorem \ref{thm:main-log-fano} and Theorem \ref{thm:main-log-fano-char-zero} provide a lot of interesting and concrete examples for $\AA^1$-connectedness. In our subsequent paper \cite{A1-connected}, we will study the properties of $\AA^1$-connectedness for general log smooth varieties, and following the work of \cite{HT-log-Fano}, an application to Zariski density of integral points over function field of curves will be considered.
\end{remark}

We next summarize the ideas used in the proof of Theorem \ref{thm:main-fano} and \ref{thm:main-log-fano}.

%------------------------------
\subsection{Log Fano complete intersections}\label{ss:log-pair}

Let $(X,D)$ be a smooth pair consisting of a smooth variety $X$ and a smooth divisor $D \subset X$. This defines the {\em divisorial log structure} on $X$:
\begin{equation}\label{equ:div-log}
\cM_{X} := \{s \in \cO_{X} \ | \ s|_{X \setminus D} \in \cO^*\}.
\end{equation}
Let $\lX = (X, \cM_{X})$ be the log scheme defined by the smooth pair $(X,D)$. In this case, $\lX$ is log smooth.  We refer to \cite{KKato} for the basic terminologies of logarithmic geometry. When there is no danger of confusion, we may use $(X,D)$ for the log scheme $\lX$ to specify the boundary $D$. 

Consider the projective space $\PP^{n}$ with homogeneous coordinates 
\[
\bar{x} = [x_{0}:\cdots: x_{n}].
\] 
Throughout this paper, we fix a sequence of non-negative integers 
\begin{equation}\label{equ:degree-sequence}
d_{1}, \cdots, d_{l}, d_{b}
\end{equation}
such that $d_b + \sum_{i=1}^l d_{i} \leq n$. We further require that $d_i$ to be positive for any $i = 1, \cdots, l$. Note that we allow $d_{b} = 0$. Choose a collection of general homogeneous polynomials in $\bar{x}$:
\[
F_{1}, F_{2}, \cdots, F_{l}, G
\]
with degrees $\deg G = d_b$ and $\deg F_{i} = d_{i}$ for all $i$. 

Let $X \subset \PP^n$ be the sub-scheme defined by $F_{i}$ for $i=1,\cdots, l$, and $D \subset X$ be the locus cut out by $G$. Since $G$ and $F_i$ are general, we may assume that $(X,D)$ is a smooth pair. We call the corresponding log scheme $\lX$ a {\em log $( d_{1}, \cdots, d_{l}; d_b)$-complete intersection}. We say $\lX$ has a {\em tame boundary} if $\ch\kk \nmid d_b$. When $d_b = 0$, $X$ is a Fano $(d_{1}, \cdots, d_{l})$-complete intersection in the usual sense.

\subsection{Separable $\AA^1$-connectedness}\label{ss:log-rc}
Let $\lcX \to \lB$ be a morphism of log schemes. A {\em stable log map} over a log scheme $\lS$ is a commutative diagram
\begin{equation}
\xymatrix{
\lC \ar[r]^f \ar[d] & \lcX \ar[d] \\
\lS \ar[r] & \lB
}
\end{equation}
such that $\lC \to \lS$ is a family of log curves over $\lS$ as defined in \cite{FKato, LogCurve}, and the underlying map $\uf$ is a family of usual stable maps to the underlying family of targets $\cX/ B$. 

The theory of stable log maps is developed by Gross-Siebert \cite{GS}, and independently by Abramovich-Chen \cite{Chen, AC}. The most important result about log maps we will need in this paper, is that they form an algebraic stack. Both \cite{GS} and \cite{Chen} assume to work over field of characteristic zero for the purpose of Gromov-Witten theory, but the proof of algebraicity works in general. Olsson's log cotangent complex \cite{logcot} provides a well-behaved deformation theory for studying stable log maps when $\lcX \to \lB$ is log smooth. 

For a log smooth scheme $\lX$ over $\spec \kk$, we write $\Omega_{\lX}$ and $T_{\lX} = \Omega_{\lX}^{\vee}$ for the log cotangent and tangent bundles respectively. Generalizing Definition \ref{def:rc}(3), we introduce the terminologies which are crucial to our construction:

\begin{definition}\label{def:log-rc}
A log scheme $\lX$ given by a divisorial log smooth pair $(X,D)$ is called {\em separably $\AA^1$-connected} (respectively {\em separably $\AA^1$-uniruled}) if there is a single-marked, genus zero log map $f: \lC/\lS \to \lX$ with $C \cong \PP^1$ and $S$ a geometric point, such that $f^*T_{\lX}$ is ample (respectively semi positive), and the tangency at the marking is non-trivial. We call such log stable map a \emph{very free} (respectively \emph{free}) {\em $\AA^1$-curve}. 
\end{definition}

\begin{remark}\label{rem:log-def}
\begin{enumerate}
\item Since the tangency at the marking is non-trivial, the image of the marked point has to lie on the boundary $D$. We refer to \cite{ACGM} for the canonical evaluation spaces of the markings. 
\item The definition of log maps allow the image of components of the source curve lies in the boundary divisor. But when $f^*T_{\lX}$ is semi positive, a general deformation yields a map with smooth source curve whose image meets the boundary divisor only at the markings, see Lemma \ref{lem:smoothing}. Thus, the above definition of $\AA^1$-uniruledness is compatible with the definition in  \cite{KM}. 

\end{enumerate}
\end{remark}

\subsection{Proof of Theorem \ref{thm:main-fano} and \ref{thm:main-log-fano}}\label{ss:method}
Same as in \cite{Zhu}, the approach we will use here is by taking degenerations. However, this time we are able to chase the deformation theory with the help of logarithmic geometry. We summarize the steps in the proof, and refer to later sections for the technical details. 

First, consider a general Fano $(d_1,\cdots,d_l)$-complete intersection $X$. We take a general simple degeneration of $X$ as in Section \ref{ss:simple-degeneration}, and obtain a singular fiber by gluing a general log Fano $(d_{1},\cdots, d_{l} - 1; 1)$-complete intersection $(X_1,D)$ and a general log Fano $(d_{1},\cdots, 1; d_{l} - 1)$-complete intersection $(X_2,D)$ along the boundary divisor $D$. 

%First we take a simple degeneration from a general Fano $(d_1,\cdots,d_l)$-complete intersection $X$ to a union of a smooth $(d_1,\cdots,d_{l-1},d_l-1)$-complete intersections $X_1$ and a smooth $(d_1,\cdots,d_{l-1},1)$-complete intersection $X_2$ with the common smooth boundary $Y$. The pairs $(X_1,Y)$ and $(X_2,Y)$ are log Fano.

\begin{obs}[See Proposition \ref{prop:fano-to-log}]\label{obs:1}
The general fiber is SRC if 
\begin{enumerate}
\item $(X_1,D)$ is separably $\A^1$-connected;
\item $(X_2,D)$ is separably $\A^1$-uniruled.
\end{enumerate} 
\end{obs}

Second, note that $D$ is a general Fano $(d_1,\cdots,d_{l-1},d_l-1,1)$-complete intersection in one-dimensional lower. We reduce both (1) and (2) in Observation \ref{obs:1} to the SRC property of the boundary $D$:

\begin{obs}[See Lemma \ref{lem:log-uniruled} and Proposition \ref{prop:log-to-usual}] \hspace{20 mm}
\begin{enumerate}
\item $(X_1,D)$ is separably $\A^1$-connected if $D$ is ample and SRC.
\item $(X_2,D)$ is separably $\A^1$-uniruled if $D$ is ample and separably uniruled.
\end{enumerate}
\end{obs}

Finally, the inductive process ends when $D$ is a projective space, which is of course separably rationally connected.

%First consider a general Fano $(d_1,\cdots,d_l)$-complete intersection $X$. Taking a simple degeneration, we show in Proposition \ref{prop:fano-to-log} that the existence of very free curves on $X$ can be deduced from the existence of very free $\AA^1$ on log Fano $(d_1,\cdots,d_l - 1; 1)$-complete intersection, and a free $\AA^1$ on log Fano $(d_1,\cdots,1; d_l - 1)$-complete intersection. 

%Now consider a log Fano complete intersection $\lX$ given by a pair $(X,D)$. Clearly the boundary $D$ is a Fano complete intersection.  In  Proposition \ref{prop:log-to-usual}, we show that the existence of very free curves on $D$ implies the existence of very free $\AA^1$ (respectively free $\AA^1$) over $\lX$ when the boundary is tame (respectively non-tame). This reduces Theorem \ref{thm:main-log-fano} to Theorem \ref{thm:main-fano}. 

%Combining with the simple degeneration argument, we reduce the existence of very free curves on general Fano $(d_1,\cdots,d_l)$-complete intersections to the existence of very free curves on general Fano $(d_1,\cdots,d_l-1, 1)$-complete intersections. By induction, both Theorem \ref{thm:main-log-fano} and \ref{thm:main-fano} follows from the fact that projective space is separably rationally connected.

\subsection*{Acknowledgments}
The authors would like to thank  Dan Abramovich for sharing his ideas, and giving many helpful suggestions on the preliminary versions of this paper. They would like to thank Chenyang Xu for suggesting this collaboration and useful discussions.  The authors are grateful to Johan de Jong and Jason Starr for useful discussions, and the anonymous referee for his/her detailed comments and suggestions on the manuscript.

%%%%%%%%%%%%%%%%%%%%Free lines and first reduction
\section{Free $\AA^1$ lines on log Fano complete intersections}\label{sec:free-line}

The following is a variation of \cite[Theorem 4.2]{Angelini}, which allows us to construct free $\AA^1$ lines explicitly on log Fano complete intersections.

\begin{lemma}\label{lem:log-tanget}
Let $X$ be a smooth $(d_1,\cdots,d_{l})$-complete intersection in $\P^n$ defined by 
$$\{F_1=\cdots=F_l=0\}$$ 
and $D = \cup_{j=1}^{k}{D_{j}}$ be a simple normal crossings divisor on $X$ with each irreducible component $D_{j}$ defined by $\{G_{j} = 0\}$ for all $j$. Let $\lX$ be the log variety associated to the pair $(X,D)$, and write $d'_{j} = \deg G_j$. Then the log tangent bundle $T_{\lX}$ is the middle cohomology of the following complex:
$$\begin{CD} \O_{X}@>A>>\O_{X}(1)^{\oplus (n+1)}\oplus \O_{X}^{\oplus l} @>B>> \sum^l_{i=1}\O_{X}(d_i) \oplus \sum_{j=1}^{k}\O_{X}(d_j),
\end{CD}$$
with the arrows defined by
\[ A = (x_0,\cdots,x_n, d_1',\cdots, d'_k)^T\] 
and 
\[
B = \left(
\begin{array}{cc}
\Jac\vec{F} & 0 \\
\Jac \vec{G} & \diag(\vec{G})
\end{array}
\right)
\]
where $\vec{F} = (F_{1}, \cdots, F_{l})$, $\vec{G} = (G_{1},\cdots, G_{k})$, $\Jac\vec{F}$ and $\Jac\vec{G}$ are the corresponding Jacobian matrices, and $\diag \vec{G}$ denotes the diagonal matrix. 

Furthermore, when $\ch \kk \nmid d'_j$ for all $j$, the log tangent bundle $T_{\lX}$ is given by the kernel of the following morphism
$$\begin{CD} \O_{X}(1)^{\oplus (n+1)} \oplus \cO_{X}^{\oplus(k-1)} @>B'>> \sum^l_{i=1}\O_{X}(d_i) \oplus \sum_{j=1}^{k}\O_{X}(d_j),
\end{CD}$$
where 
\[
B' = \left(
\begin{array}{cc}
\Jac\vec{F} & 0 \\
\Jac \vec{G} & \diag(G_1, \cdots, G_{k-1})
\end{array}
\right)
\]
\end{lemma}
\begin{proof}
In case $\vec{F} = 0$, the result is proved in \cite[Theorem 4.2]{Angelini}. In case $\vec{G} = 0$ and $\vec{F}$ is non-trivial, the tangent bundle $T_{X}$ is given by the middle cohomology of the following:
\[
\cO_{X} \stackrel{A}{\longrightarrow} \cO_{X}(1)^{\oplus(n+1)} \stackrel{\Jac\vec{F}}{\longrightarrow} \sum_{i=1}^{l}\cO_{X}(d_{i}).
\]
Now the statement follows from combining the above sequence with \cite[Theorem 4.2]{Angelini}.
\end{proof}

% \left({\begin{array}{cccc}
%\partial_0F_1 & \cdots & \partial_{n}F_1 & 0\\
%\vdots & & \vdots &\vdots \\
%\partial_0F_c & \cdots & \partial_{n}F_c & 0\\
%\partial_0G & \cdots & \partial_{n}G & G
%\end{array}}\right).

\begin{proposition}\label{prop:line}
Let $(X,D)$ be a general log Fano $(d_1,\cdots,d_{l}; d_b)$-complete intersection in $\P^n$ with $e=\sum^l_{i=1}d_i+d_b \leq n$. If $\ch \kk\nmid d_b$, then the pair $(X,D)$ is separably log-uniruled by lines. Furthermore, the restriction of the log tangent bundle to a general log free line has splitting type 
\[\O(1)^{\oplus{(n+1-e)}} \oplus \O^{\oplus({e-l-1})}.\]
\end{proposition}

\proof By the log deformation theory, it suffices to produce a pair $(X,D)$ log-smooth along a line such that the restriction of the log tangent bundle is semi positive. 

Let $L$ be the line defined by 
\[\{x_2=\cdots=x_{n}=0\}.\] 
For simplicity, we introduce $m_{j} = \sum_{i=1}^j d_i$, and set $m_0 = 0$ and $m_{l+1} = e$. Choose the following homogeneous polynomials:
\begin{equation}
\begin{array}{lll}
F_{i} & = & x_{m_{i-1} + 2}\cdot x_{0}^{d_i - 1} + x_{m_{i-1} + 3}\cdot x_1 \cdot x_{0}^{d_i - 2} + \cdots + x_{m_i+1}\cdot x_1^{d_{i}-1} \\
G & = & x_{1}^{d_b} + x_{m_l + 2}\cdot x_{1}^{d_b - 2}\cdot x_0 + \cdots + x_{m_{l+1}}\cdot x_{0}^{d_b-1}.
\end{array}
\end{equation}
for $i \in \{1,\cdots, c\}$. Note that when $d_b = 1$, we get
\begin{equation}
G = x_1.
\end{equation}
and all $F_i$ remain the same. We then check that
\begin{enumerate}
 \item $L$ lies in the smooth locus of $X$;
 \item $L$ intersects $D$ only at the point $[x_0:x_1] = [1:0]$, which is a smooth point of $D$.
% \item the log tangent direction of the line at the marking does not contained in the normal direction of the divisor.\Qile{Need to make this precise}
\end{enumerate}

By Lemma \ref{lem:log-tanget}, we have a short exact sequence of sheaves over $L$ after twisting down by $\O(-1)$:
\[
0 \to T_{X^\dagger}|_L(-1) \to \cO^{\oplus(n+1)} \stackrel{B'}{\longrightarrow}  \sum^c_{i=1}\O(d_i-1) \oplus \O(d_b-1)\to 0.
\]

Note that $ T_{X^\dagger}|_L$ is semi positive if and only if $H^1(L,T_{X^\dagger}|_L(-1))=0$. By the long exact sequence of cohomology, it suffices to show that

\[
H^0(\cO^{\oplus(n+1)}) \stackrel{B'}{\longrightarrow}  H^0(\sum^c_{i=1}\O(d_i - 1) \oplus \O(d - 1))
\]
is surjective. This follows from the assumption $\ch \kk\nmid d_b$, and the choice of polynomials $F_{i}$ and $G$. Finally, since $ T_{X^\dagger}|_L$ is a subsheaf of $\O(1)^{\oplus{n+1}}$, the splitting type of the log tangent bundle on the line is as desired. This finishes the proof.
%\Qile{Needs to be clarified according to the referee report (15).}
\qed

%%%%%%%%Induction
\section{Reduction to log Fano varieties via degeneration}\label{sec:induction}
\subsection{Simple degeneration}\label{ss:simple-degeneration}

Consider a log smooth morphism of fine and saturated log schemes $\pi: \lX_0 \to p^{\dagger}$ with $p^{\dagger}$ the standard log points, i.e. $\ocM_{p^{\dagger}} = \cM_{p^{\dagger}}/\kk^* \cong \NN$. 

\begin{definition}\label{def:simple-degeneration}
We call such log smooth morphism $\pi: \lX_0 \to p^{\dagger}$ a {\em simple degeneration} if  the underlying space $X_0$ is given by two smooth varieties $Y_{1}$ and $Y_{2}$ intersecting trasversally along a connected smooth divisor $D$. 
\end{definition}

\begin{remark}\label{rem:simple-deg-can}
By \cite{LogSS}, any log smooth morphism $\pi$ as above which is a simple degeneration, admits a canonical log structure $\tilde{\pi}: \tilde{X}^{\dagger} \to p^{\dagger}$ and a morphism $g: p^{\dagger} \to p^{\dagger}$ such that $\pi$ is the pull-back of $\tilde{\pi}$ along $g$.
\end{remark}

Assume we are in the situation of Section \ref{ss:log-pair}. We fix a smooth $(d_1,\cdots,d_{l-1})$-complete intersection $W$ of codimension $l-1$ in $\P^n$ cut out by $F_1,\cdots,F_{l-1}$. Let $G_l$ be the product $G_1 G_2$ of two homogeneous polynomials of degree $a$ and $d_l-a$ and let $F_l$ be a homogeneous polynomial of degree $d_l$.

Consider the pencil of divisors in $Z\subset W\times\A^1$ defined by $\{t \cdot F_l+ G_l=0\}$. Let $\pi: Z \to \A^{1}$ be the projection to the second factor. 

For a general choice of $F_1,\cdots,F_l, G_1$, and $G_{2}$, there exists an open neighbourhood $U\subset\A^1$ of $0$ satisfying the following properties:
\begin{enumerate}
\item $\pi: \cX:=\pi^{-1}U\rightarrow U$ is a flat family of $(d_1,\cdots,d_l)$-complete intersections in $\P^n$;
\item the general fibers $\cX_t$ are smooth;
\item the special fiber $\cX_0$ is a union of a smooth $(d_1,\cdots,d_{l-1},a)$-complete intersection $X_1$ and a smooth $(d_1,\cdots,d_{l-1},d_l-a)$-complete intersection $X_2$;
\item the intersection $D$ of $X_1$ and $X_2$ is a smooth $(d_1,\cdots,d_{l-1},a, d_l-a)$-complete intersection;
\item the singular locus of the total space is given by the base locus $$\{F_l=0\} \cap D,$$ and is of codimension one in $D$.
\end{enumerate}
Let $\cX^\circ$ be the complement of $\{F_l=0\} \cap D$ in $\cX$. $\cX^{\circ}$ is the smooth locus of the total space $\cX$ in the usual sense. Consider the canonical divisorial log structure $\cM_{\cX^{\circ}}$ associate to the pair $(\cX^{\circ}, \partial \cX^{\circ}:= \pi^{-1}(0))$, and the log structure $\cM_{\AA^1}$ associated to $(\AA^1, 0)$. Then we have a morphism of log schemes
\begin{equation}\label{equ:target-degeneration}
\pi^{\dagger}: (\cX^{\circ}, \cM_{\cX^{\circ}}) \to (\AA^{1}, \cM_{\AA^{1}}).
\end{equation}
Write 
\begin{equation}\label{equ:simple-degeneration}
\pi^{\dagger}_{0}: Y^{\dagger} \to 0^{\dagger}
\end{equation} 
for the fiber over $0 \in \AA^1$. The closure of the underlying scheme $Y$ of $Y^\dagger$ is given by $X_{1}\cup X_{2}$.

\begin{lemma}\label{lem:degeneration}
The log map $\pi^{\dagger}$ is a log smooth morphism of fine and saturated log schemes and the central fiber $Y^{\dagger}$ is a simple degeneration. In particular, the general point of $D$ lies in the log smooth locus of $\pi^\dagger$.
\end{lemma}
\begin{proof}
This follows since the pair $(\cX^{\circ}, \partial \cX^{\circ})$ over $\AA^1$ is a simple normal crossings degeneration.
\end{proof}

\subsection{The gluing construction}\label{ss:first-gluing}

%\Yi{why $\uf$? not simply $f$? we have used $f^\dagger$ for extra log structure on the morphism.}\Qile{we will need to use $f^*$ quite often. If we use $f^{\dagger}$, then we need to complicates the notation by using $(f^{\dagger})^{*}$.}
\begin{lemma}\label{lem:simple-degenerate-gluing}
Let $\pi: Y^{\dagger} \to p^{\dagger}$ be a simple degeneration with its canonical log structures as in Definition \ref{def:simple-degeneration}. Let $\uf: C \to Y$ be a usual genus zero stable map with $C$ given by two irreducible components $C_{1}$ and $C_{2}$ glued along a node $x \in C$. Further assume that $f^{-1}(D) = x$ with the same contact orders $c$ on each components. Then there is a stable log map given by the following diagram:
\begin{equation}\label{diag:simple-deg-gluing}
\xymatrix{
\lC \ar[rr]^{f} \ar[d] && \lY \ar[d] \\
p^{\dagger} \ar[rr]^{u} && p^{\dagger}.
}
\end{equation}
over the underlying stable map $\uf$, such that the log structure associated to $\lC \to p^{\dagger}$ is the canonical one as in \cite{FKato, LogCurve}. Furthermore, on the level of characteristics $\bar{u}^{\flat}: \NN \to \NN$ is given by multiplication by $c$.
\end{lemma}
\begin{proof}
This follows from the construction in \cite[Section 5.2.3]{Kim}.
\end{proof}

We next consider a log smooth variety $\lX$ given by a smooth variety $X$ and a smooth divisor $D \subset X$. Consider the $\PP^1$-bundle $\PP := \PP(N_{D/X}\oplus\cO_{D})$ with two disjoint divisors $D_{0} \cong D_{\infty} \cong D$ such that $N_{D_{0}/\PP} \cong N^{\vee}_{D/X}$ and $N_{D_{\infty}/\PP}\cong N_{D/X}$. Gluing $\PP$ and $X$ by identifying $D_{0}$ with $D$, we obtain a scheme $Y$. By \cite[Theorem 11.2]{FKato}, there is simple degeneration with the central fiber $\pi: \lY \to p^{\dagger}$  a log smooth simple degeneration  as in Lemma \ref{lem:simple-degenerate-gluing}.  By \cite[Proposition 6.1]{GS}, there is a log map $g: \lY \to \lX$ contracting the $\PP^1$-bundle $\PP$ to the divisor $D$. 

\begin{lemma}\label{lem:relative-smooth-out}
Consider a genus zero stable map
\[
\uf: C \to D
\]
such that $C \cong \PP^1$ and $\deg (N_{D/X})|_{C} = c \geq 0$. Then there is a log map $f: \lC \to \lX$ over $\uf$ with a unique marking $\sigma \in \lC$ of contact order $c$. 
\end{lemma}
\begin{proof}
Pick an arbitrary point $\sigma \in C$, and fix an isomorphism
\[
(N_{D/X})|_{C} \cong \cO_C(c \cdot \sigma).
\]
Choose a section $s \in H^0(\cO_{C}(c\cdot \sigma))$ with a zero of order $c$ at $\sigma$. Thus, the section $s$ defines a map 
\[
\uf': C \to \PP
\]
which is tangent to $D_{\infty}$ at $\sigma$ only of contact order $c$, and does not meet $D_{0}$. By \cite[Section 5.2.3]{Kim}, there is a log map 
\[
\xymatrix{
C^{\dagger} \ar[r]^{f'} \ar[d] & \lY \ar[d] \\
p^{\dagger} \ar[r] & p^{\dagger}.
}
\]
Consider the morphism of log schemes $g: \lY \to \lX$. Now the composition $f:= g\circ f'$ defines the log map we want.
\end{proof}

\begin{lemma}\label{lem:relative-gluing}
Notations as above, consider a genus zero stable map
\[
\uf: C \to X
\]
such that
\begin{enumerate}
 \item $C$ has two irreducible components $C_{1}$ and $C_{2}$ meeting at the node $x$;
 \item $\uf|_{C_{1}}$ only meets $D$ at the node $x$ of contact order $c_{1}$;
 \item $\uf(C_{2}) \subset D$, and $\deg (\uf^*(N_{D/X}))_{|_{C_{2}}} = c_{2}$.
\end{enumerate}
Assume that $c_{1} + c_{2} \geq 0$. Then there is a stable log map $f: \lC/p^{\dagger} \to \lX$ over $\uf$ with a single marked point $\sigma \in C_{2}$ of contact order $c_{1} + c_2$. Furthermore, the log structure on $\lC \to p^{\dagger}$ is the canonical one as in \cite{FKato, LogCurve}.
\end{lemma}
\begin{proof}
We define a morphism of sheaves over $C_2$:
\[
\cO\oplus \cO(-c_2) \to \cO(c_1)
\]
where the arrow $\cO \to \cO(c_1)$ is defined by the effective divisor $c\cdot x$, and $\cO(-c_2) \to \cO(c_1)$ is defined by the effective divisor $(c_1 + c_2)\cdot \sigma$. This defines a morphism $C_{2} \to \PP$ tangent to $D_{\infty}$ and $D_{0}$ at $\sigma$ and $x$ with contact orders $c_{1}+c_{2}$ and $c_{1}$ respectively. We are in the situation of Lemma \ref{lem:simple-degenerate-gluing}. Thus, there is a stable log map $f'$ as in (\ref{diag:simple-deg-gluing}) over the underlying map $\uf$. The composition of $f:= f'\circ g$ yields the stable log map as in the statement.
\end{proof}

\begin{remark}
In Lemma \ref{lem:relative-gluing}, the marking $\sigma$ can be removed if $c_1 + c_2 = 0$.
\end{remark}

For the reader's convenience, we include the following result, which is known to experts:

\begin{lemma}\label{lem:smoothing}
Consider a genus zero log map 
\[
\xymatrix{
\lC \ar[r]^f \ar[d] & \lcX \ar[d] \\
\lS \ar[r] & \lB
}
\]
where the underlying $S$ is a geometric point, $\lcX \to \lB$ is a log smooth family, and the log structure of $\lB$ over the generic point is trivial. 
Assume $f^{*}T_{\lcX/\lB}$ is semi-positive. Let $f'$ be a general smoothing of $f$. Then 
\begin{enumerate}
 \item the source curve of $f'$ is irreducible;
 \item the map $f'$ only meets the locus $\partial \lcX$ with non-trivial log structure at the marked points.
\end{enumerate}
\end{lemma}
\begin{proof}
Let $\fK$ be the moduli space of stable log maps, and $\fM$ be the moduli space of genus zero pre-stable curves with its canonical log structure. Then the semi-positivity of  $f^{*}T_{\lcX/\lB}$ implies the morphism of usual algebraic stacks
\begin{equation}\label{equ:removing-map}
\fK \to \cL og_{\fM\times \lB}
\end{equation}
is smooth at the point $[f] \in \fK$, see \cite[Section 2.5]{Chen}. Here $\cL og_{\fM\times \lB}$ is Olsson's log stack parameterizing log structures over $\fM\times \lB$, see \cite{Olsson-ENS}. By assumption $\cL og_{\fM\times \lB}$ contains an open dense sub-stack with trivial log structures. Thus, a general deformation $f'$ satisfies the conditions in the statement, see for example \cite[Section 3.2]{Chen}.
\end{proof}

% with $\lB$ given by either $\spec \kk$ or a connected smooth curve $B$ with a point $0 \in B$. 

\subsection{From Fano to log Fano via a simple degeneration}

\begin{proposition}\label{prop:fano-to-log}
Notations as in Lemma \ref{lem:degeneration}. Consider the two log Fano varieties $\lX_{i}$ associated to $(X_{i}, D)$ for $i=1,2$ as in Section \ref{ss:simple-degeneration}. If $\lX_1$ is separably $\AA^1$-connected, and $\lX_2$ is separably $\AA^1$-uniruled, then general fibers of (\ref{equ:target-degeneration}) are separably rationally connected.
\end{proposition}
\begin{proof}
By assumption, we may take a very free $\AA^1$ curve $f_1:\lC_1 \rightarrow \lX_1$, and a free $\AA^1$ curve $f_2: \lC_2 \rightarrow \lX_2$.  Write $\sigma_i$ for the unique marking on $\lC_i$ for $i=1,2$.  Note that the marking of the free log map sweep out general points on the boundary divisor \cite[ Corollary 5.5(3)]{KM}. We may assume that  $f_1(\sigma_1)=f_2(\sigma_2)$, and $f_i(\lC_i \setminus \{\sigma_i\}) \cap D_i = \emptyset$ for $i=1,2$.

After composing $f_i$ with some generically \'etale multiple cover by rational curves ramified at $\sigma_i$, we may assume that $f_1$ and $f_2$ have the same contact orders along the common boundary. By Lemma \ref{lem:simple-degenerate-gluing}, we may glue $f_1$ and $f_2$ along the markings, and obtain a stable log map $f: \lC \to \lX$ where the underlying curve $C$ is a rational curve with one node obtained by gluing $C_{1}$ and $C_{2}$ along the markings.  

Since the pullback of the log tangent bundles $f_1^*T_{\lX_{1}}$ and $f_2^*T_{\lX_2}$ are at least semi positive, there exists a smoothing $f'$ of $f$ to the general fiber of the one parameter degeneration by Lemma \ref{lem:smoothing}. Since $f_1^*T_{\lX_{1}}$ is ample by assumption, a general smoothing $f'$ is very free.
\end{proof}

%%%%%%%%%%%%%%%%%%%
\section{Reduction to the Fano boundary}

\subsection{Separably $\A^1$-uniruledness}
The following can be found in \cite[5.2]{KM}. For completeness, we include the proof here.

\begin{lemma}\label{lem:divisor-ambient}
Let $X^{\dagger}$ be a log smooth scheme given by a normal crossing pair $(X,D=\sum_{i=1}^k D_i)$. Then we have an exact sequence
\begin{equation}\label{equ:divisor-ambient-tangent}
0 \to \cO_{D_i} \to T_{X^\dagger}|_{D_i} \to T_{D_i^\dagger} \to 0,
\end{equation}
where $D_i^\dagger$ is given by the pair $(D_i, \sum_{j\neq i}D_j|_{D_i})$.
\end{lemma}
\begin{proof}
Write $Z^{\dagger}$ to be the log scheme given by $(X,\sum_{j \neq i }D_{j})$. Consider the exact sequence over $X$:
\[
0 \to \Omega_{Z^{\dagger}} \to \Omega_{\lX} \to \cO_{D_i} \to 0
\]
Applying $\otimes\cO_{D_i}$ to the above sequence, we have
\[
0\to \Tor_{1}^{\cO_{X}}(\cO_{D_{i}}, \cO_{D_{i}}) \to \Omega_{Z^{\dagger}}|_{D_{i}} \to \Omega_{\lX}|_{D_{i}} \to \cO_{D_{i}} \to 0.
\]
Note that $\Tor_{1}^{\cO_{X}}(\cO_{D_{i}}, \cO_{D_{i}}) = N^{\vee}_{D_{i}/X}$. Now the statement follows from taking the dual of the above exact sequence. 
\end{proof}

\begin{lemma}\label{lem:log-uniruled}
Notations as in Lemma \ref{lem:divisor-ambient}, assume that there exists $D_i^\dagger$ such that it is separably $\A^1$-uniruled, and $\deg f^*N_{D_{i}} > 0$ for some free $\AA^1$-curve $f: Z^{\dagger} \to D_i^\dagger$, then $\lX$ is separably $\A^1$-uniruled.

When $D$ is a smooth irreducible ample divisor which is separably uniruled, $X^\dagger$ is separably $\A^1$-uniruled.
\end{lemma}
\begin{proof}
We will give a proof of the first statement. The second statement can be proved similarly.

By the assumption and log deformation theory, we may choose a free $\AA^1$-curve $f': Z^{\dagger} \to D_{i}^{\dagger}$ such that 
\begin{enumerate}
 \item the underlying source curve $Z \cong \PP^1$ is irreducible with a unique marking $\sigma \in Z$;
 \item $f'(Z) \not\subset D_{j}$ for any $j \neq i$.
 \item $\deg f'^*N_{D_i|X}>0$.
\end{enumerate}

For each $j$, consider another log scheme $\lX_j$ given by the pair $(X, D_j)$. By Lemma \ref{lem:relative-smooth-out}, we could lift $f'$ to a genus zero stable log map $f''_i: C^{\dagger}/S^{\dagger} \to \lX_i$ with the unique marking $\sigma$. Since the image of $f''_i$ is not contained in $D_j$ for any $j \neq i$, we obtain a stable log map $Z^{\dagger} \to \lX_{j}$ with the same underlying map given by $f'$, and a unique marking $\sigma$, possibly having the trivial contact order. Consider the composition
\[
f''_j: \lC \to Z^{\dagger} \to \lX_j.
\]
Now the product
\[
f:=\prod_{j=1}^k f''_j : \lC \to \lX_1\times_{X}\cdots\times_{X}\lX_k \cong \lX
\]
defines an $\AA^1$-curve in $\lX$. Using Lemma \ref{lem:divisor-ambient}, we could check that $f$ is a free $\AA^1$-curve. This finishes the proof.
\end{proof}

\subsection{Separably $\A^1$-connectedness}\label{ss:log-fano}
The goal of this section is to prove the following:

\begin{proposition}\label{prop:log-to-usual}
Let $\lX$ be a general log Fano $( d_{1}, \cdots, d_{l}; d_b)$-complete intersection given by the pair $(X,D)$ as in Section \ref{ss:log-pair}. If $D$ is separably rationally connected and $\ch \kk \nmid d_b$, then $\lX$ is separably $\AA^1$-connected.

%$X$ is separably log rationally connected. 
\end{proposition}

%Assume we are in the situation of Proposition \ref{prop:log-to-usual}. Let $\uf_1: C_1 \to D$ be a very free curve on $D$. By Lemma \ref{lem:relative-smooth-out}, we have a stable log map $f_1: \lC_1 \to \lX$ over $\uf_1$. Since $\uf_{1}^*T_{D}$ is positive, by Lemma \ref{lem:divisor-ambient} $f_{1}$ is at least a free $\AA^1$. This proves the first statement in Proposition \ref{prop:log-to-usual}.

\begin{proof}
Choose a very free rational curve $\uf_1: C_1 \cong \PP^1 \to D$. Let $\sigma, \sigma_1$ be two general points on $C_1$. Since $\ch \kk \nmid d_b$, we may choose a log free line $f_2: \lC_2 \to \lX$ constructed in Proposition \ref{prop:line} with the unique marking $\sigma_2$ having image $\uf_1(\sigma_1)$. By Lemma \ref{lem:relative-gluing}, we may glue $\uf_1$ and $\uf_2$ by identifying $\sigma_{1}$ and $\sigma_2$, and obtain a stable log map $f:C^\dagger\rightarrow X^\dagger$ with one marking $\sigma$ and one node $p$.

If we restrict (\ref{equ:divisor-ambient-tangent}) to $C_1$, there are two possibilities:
\begin{enumerate}
\item $T_{X^\dagger}|_{C_1}$ is ample.
\item $T_{X^\dagger}|_{C_1}$ is a trivial extension of $T_D|_{C_1}$ by $\O_{C_1}$.
\end{enumerate}

In the first case, a general smoothing $f$ is very free by Lemma \ref{lem:smoothing}. In the second case, $T_{X^\dagger}|_{C_1}$ is only semi positive. 

Consider the composition

\begin{equation} \label{equ:curve-target-boundary}
\xymatrix{
T_{\lC_2}|_p \ar[r]^{d f_2} & T_{X^\dagger}|_p \ar[r]^\delta & T_{D}|_p.\\%@>\delta>> T_{D}|_p.
}
\end{equation}  

\begin{lemma}\label{lem:tame-non-vanishing}
The push-forward morphism $\mbox{d} f_2$ is injective when $\ch \kk \nmid d_b$. 
\end{lemma}

\proof It suffices to show that pullback morphism $(\mbox{d} f_2)^{\vee}:\Omega_{X^\dagger}|_p\rightarrow \Omega_{C^\dagger}|_p$ is surjective. We check this using a local computation. Locally at $p$, there is a log $1$-form $dg/g$ where $g$ is the defining equation of the boundary. Since the image of $C_2$ is a log free line, $(\mbox{d} f_2)^{\vee} (dg/g)=d_b \cdot dt/t\neq 0$. \qed

\begin{lemma}\label{lem:zero-projection}
The composite morphism (\ref{equ:curve-target-boundary}) is a zero morphism.
\end{lemma}
\begin{proof}
Applying Lemma \ref{lem:divisor-ambient} to both $\lC_2$ and $\lX$, and restricting to $\sigma$, we have
the commutative diagram:
\[
\xymatrix{
0 \ar[r] & \kk_{\sigma_2} \ar[r]^{\cong} \ar[d] &  T_{\lC_2}|_{\sigma_2} \ar[d] \ar[r] & 0 \\
0 \ar[r] & \cO_{D}|_{\sigma_2} \ar[r] & T_{\lX}|_{\sigma_2} \ar[r] &  T_{D}|_{\sigma_2} \ar[r] & 0
}
\]
The statement then follows.
\end{proof}

Let $E$ be the codimension one vector subspace in $T_{X^\dagger}|_p$ which corresponds to $T_D$. To make a log very free curve, it suffices to increase the positivity outside $E$. By Proposition \ref{prop:line}, the splitting type of $T_{X^\dagger}|_{C_2}$ is $\O(1)^{\oplus{(n+1-e)}} \oplus \O^{\oplus({e-l-1})}$. Let $E'$ be the canonical subspace of $T_{X^\dagger}|_p$ which corresponds to the factor $\O(1)^{\oplus{(n+1-e)}}$. By Lemma \ref{lem:tame-non-vanishing}, $E'$ contains the log tangent direction $T_{\lC_2}|_p$.

By Lemma \ref{lem:zero-projection}, $E'$ as a vector subspace in $f^*T_{X^\dagger}|_{p}$ is contained in the kernel of $T_{\lX}|_p \to T_{D}|_p$. Since $E$ is of codimension one, the two vector subspaces $E'$ and $E$ span $T_{X^\dagger}|_p$. 

Since $f$ is unobstructed with the canonical log structure $\lC \to p^{\dagger}$ on the source log curve, this implies that the composition
\[
\fK \to \cL og_{\fM} \to \fM
\]
is smooth at the point $[f]$. Here the first arrow is given by (\ref{equ:removing-map}) with $B^{\dagger}$ a geometric point with the trivial log structure.  We may thus take a general smoothing of $f$ with the total space smooth. Proposition \ref{prop:log-to-usual} then follows from Proposition \ref{prop:positivity} below.
\end{proof}

\begin{proof}[Proof of Theorem \ref{thm:main-log-fano-char-zero}]
By adjunction, $D$ is Fano, and hence separably rationally connected in characteristic zero. We may then choose a very free rational curve $\uf: C \to D$ through general points of $D$. Now the theorem is proved by gluing $\uf$ with a free $\AA^1$-curve in $(X,D)$ with sufficiently large intersection number with $D$, and applying the same argument as in Proposition \ref{prop:log-to-usual}.
\end{proof}

\begin{proof}[Proof of Corollary \ref{cor}]
By Lemma \ref{lem:log-uniruled} and Theorem \ref{thm:main-log-fano-char-zero}, it suffices to show that there exists a free rational curve $f:\P^1\to D$ such that $\deg f^*N_D>0$. Indeed, by adjunction formula, $D$ is Fano, hence rationally connected. Let $E\subset D$ be the effictive divisor determined by $N_D$. A very free rational curve passing through a point in $E$ but not lying on $E$ will do the job.
\end{proof}

\subsection{A result from elementary transform}

\begin{construction}\label{cons}
Let $C$ be the union of two irreducible rational curves $C_1$ and $C_2$ glued at a node $p$. Let $q:\cC  \rightarrow T$ be a smoothing of $C$ with $C$ the fiber over $0 \in T$. Assume that the total space $\cC$ is a smooth surface. Let $s_1$ and $s_2$ be two sections of $q$ both of which specialize to two distinct points $y_1, y_2$ on $C_1$. 
Consider a locally free sheaf $\E$ of rank $r$ on $\cC$, satisfying the following property:
\begin{enumerate}
\item $\E|_{C_1}$ is isomorphic to $\O\oplus \F$, where $\F$ is a positive sub-bundle. Let $E$ be the canonical codimension one subspace of $\E|_p$ which corresponds to $\F$.
\item $\E|_{C_2}\cong \T\oplus\O^{\oplus r-k}$, where $1\le k\le r$ and $\T$ is positive. Let $E'$ be the caonical subspace of $\E|_p$ which corresponds to $\T$.
\item $E'$ and $E$ span $\E|_p$.
\end{enumerate}
\end{construction}

Consider the following composition:
\begin{equation}r:\E^\vee\rightarrow \E^\vee|_{C_2}\rightarrow \T^\vee\end{equation}

Clearly $r$ is surjective. Let $\K^\vee$ be the kernel of $r$, i.e., the elementary transform of $\E^\vee$ along $\T^\vee$.  Consider the induced exact sequence:
\begin{equation}\label{dual}
0 \to \K^\vee \to \E^\vee \to \T^\vee \to 0
\end{equation}
Dualizing the above short exact sequence over $\cC$, we get a long exact sequence
\[
0 \to \hom_{\O_\cC}(\cT^\vee,\O_{\cC})  \to \E \to \K \to \underline{\ext}^1_{\O_{\cC}}(\T^\vee, \O_{\cC}) \to 0.
\]

The first term vanishes because it is the dual of a torsion sheaf. The last term is isomorphic to $\T\otimes_{\O_{C_2}}\O_{C_2}(C_2)$ by \cite[A3.46 b]{Eisenbud} and 
\[\ext^1_{\O_\cC}(\O_{C_2},\O_\cC)\cong \O_{C_2}(C_2).\] 
Thus we obtained a short exact sequence
\begin{equation}\label{undual}
0 \to \E  \to \K \to \T\otimes_{\O_{C_2}}\O_{C_2}(C_2)  \to 0.
\end{equation}

%Since the surface $\cC$ has $A_n$-singularities at $p$, we have
%\[
%\underline{\ext}^1_{\cO_{\cC}}(\cO_{C_2}, \cO_{\cC}) \cong \cO_{C_2}(-p)\oplus T
%\]
%where $T$ is a torsion sheaf supported on $p \in C_2$. \Qile{Is this exactly correct? I feel this is the case, but I couldn't figure it %out!} 

\begin{lemma}\label{teeth+}
$h^1(C_2, \K|_{C_2}(-p))=0$.
\end{lemma}

\proof Restricting the short exact sequence (\ref{dual}) to $C_2$, applying the functor $\Hom_{\O_{C_2}}(\  * \  ,\O_{C_2})$, and combining with (\ref{undual}),  we obtain
\[
0 \to \T \to \E|_{C_2}  \to \K|_{C_2} \to  \cT(-p)  \to 0.
\]
The quotient bundle $\E|_{C_2}/\T$ is a trivial vector bundle and the last term of the exact sequence is isomorphic to $\T(-p)$. In particular, we have 
\[
0 \to \O_{C_2}(-p)^{\oplus (r-k)}   \to \K|_{C_2}(-p) \to \T(-2p)  \to 0.
\]
The lemma follows from the vanishing of $H^1$ of the first and the third term of the above sequence.\qed 

\begin{lemma}\label{handle+}
$h^1(C_1,\K|_{C_1}(-y_1-y_2))=0$.
\end{lemma}

\proof Restricting the short exact sequence (\ref{dual}) to $C_1$, we get 
\[
\K^\vee|_{C_1}  \to \E^\vee|_{C_1} \to \T^\vee|_{C_1} \to 0.
\]
The above sequence is also left exact. Indeed, since $\T\otimes_{\O_{C_2}}\O_{C_2}(C_2)|_{C_1}$ is torsion, by restricting (\ref{undual}) to $C_1$ and taking the dual over $C_1$,  we have the injection from $\K^\vee|_{C_1}$ to $\E^\vee|_{C_1}$. 

In other words, the vector bundle $\K|_{C_1}$ is the elementary transform of $\E|_{C_1}$ along $p$ with the specific subspace $E'$. By condition (3) of the construction, $E'$ does not lie in $\F$ at $p$. This implies that $\K$ is ample on $C_1$. The statement follows.

\begin{proposition}\label{prop:positivity}
With the same notations and constructions as above, the restriction of $\E$ to a general fiber $\cC_t$ is positive.
\end{proposition}

\begin{proof} By the construction, we know $\K|_{\cC_t}$ is isomorphic to $\E|_{\cC_t}$. Since $\K$ is locally free on $\cC$, it is flat over $T$. By upper semicontinuity, it suffices to show that $h^1(C,\K(-y_1-y_2))=0$. This follows from the above two lemmas and the restrictions of the short exact sequence.
\end{proof}

%\bibliographystyle{amsalpha}             % (uses file "plain.bst")
%\bibliography{Fano} 

%%%%%%%%%%%%%%%%%

\end{document}